\documentclass{amsart}

\usepackage{a4}
\usepackage{graphicx}
\usepackage{verbatim}
\usepackage{amssymb} 
\usepackage{multirow}
\usepackage{imakeidx}

\usepackage[utf8]{inputenc}

\usepackage{multicol}
\usepackage{mathdots}

\newtheorem{theorem}{Theorem}[section]
\newtheorem{lemma}[theorem]{Lemma}

\newtheorem{prop}[theorem]{Proposition}
\newtheorem{coro}[theorem]{Corollary}
\theoremstyle{definition}
\newtheorem{definition}[theorem]{Definition}

\theoremstyle{remark}
\newtheorem{remark}[theorem]{Remark}

\numberwithin{equation}{section}
\makeatletter
\newcommand{\leqnomode}{\tagsleft@true}
\makeatother

\makeindex
\begin{document}

\begin{abstract}
   We are interested in the classification of left-invariant symplectic structures on Lie groups. Some classifications are known, especially in low dimensions. In this paper we establish a new approach to classify (up to automorphism and scale) left-invariant symplectic structures on Lie groups. The procedure is based on the moduli space of left-invariant nondegenerate $2$-forms. Then we apply our procedure for two particular Lie groups of dimension $2n$ and give classifications of left-invariant symplectic structures on them. 
\end{abstract}
\title[A classification of left-invariant symplectic structures]{A classification of left-invariant symplectic structures on some Lie groups}

\author{Luis Pedro Castellanos Moscoso}
\address[L. Castellanos]{Department of Mathematics, Osaka City University, 
Sugimoto, Sumiyoshi-ku, Osaka City, Japan 585-8585}
\email{caste3.1416@gmail.com}
\thanks{
This work was partly supported by Osaka City University Advanced 
Mathematical Institute (MEXT Joint Usage/Research Center on Mathematics 
and Theoretical Physics JPMXP0619217849). The second author was supported by JSPS KAKENHI Grant Number 19K21831. 
} 

\author{Hiroshi Tamaru}
\address[H.~Tamaru]{Department of Mathematics, Osaka City University, 
Sugimoto, Sumiyoshi-ku, Osaka City, Japan 585-8585} 
\email{tamaru@sci.osaka-cu.ac.jp}







\maketitle


\section{Introduction}
In geometry it is an important problem to study whether a given manifold admits some nice geometric structures. In the setting of Lie groups, it is natural to ask about the existence of left-invariant structures. A \textit{symplectic Lie group} is a Lie group $G$ endowed with a left-invariant symplectic form $\omega$ (that is, a nondegenerate closed $2$-form). The geometry of symplectic Lie groups is an active field of research. There are many interesting results on the structure of symplectic Lie groups and some classifications in low dimensions, but the general picture is far from complete. Conjectures about the existence of isotropic subgroups are still unsolved. Some nice known results include

\begin{enumerate}
    \item Unimodular symplectic Lie groups are solvable. All symplectic Lie groups of dimension $4$ are solvable (\cite{SymplecticHom}).
    \item Some of the known classification in low dimensions include: complete classification for the $4$-dimensional case (\cite{Ovando}), filiform Lie algebras up to dimension $10$  (\cite{Filiform}), most of solvable Lie algebras up to dimension $6$ (\cite{Nilpotent},\cite{Solvable}). 
    \item Completely solvable symplectic Lie groups have Lagrangian subgroups. If dimension $\leq 6$, every symplectic Lie group has a Lagrangian
subgroup. All nilpotent symplectic Lie groups of dimension $\leq 6$ have a Lagrangian normal subgroup (\cite{SymplecticLieGroups}).
\end{enumerate}

In \cite{TakahiroHashinaga2016} we can find a novel method to find nice (e.g., Einstein or Ricci soliton) left-invariant Riemannian metrics. The method is based on the moduli space of left-invariant Riemannian metrics on a Lie group $G$. Recall that there is a correspondence between left-invariant metrics on $G$ and inner products on its Lie algebra $\mathfrak{g}$.  The moduli space is defined as the orbit space of the action of $\mathbb{R}^{\times} \textrm{Aut} (\mathfrak{g})$ on the space $\tilde{\mathfrak{M}}(\mathfrak{g})$ of inner products on $\mathfrak{g}$. An expression of the moduli space derives a \textit{Milnor-type theorem}. Milnor-type theorems, first introduced in the aforementioned paper, can be thought as a generalization of so called Milnor frames introduced by Milnor in his famous paper \cite{Milnorpaper}.  In the setting of Riemannian metrics a Milnor-type theorem gives an orthonormal basis for each inner product $\langle, \rangle$ in $\mathfrak{g}$, such that the bracket relations between the elements of the basis are given in terms of a small number of parameters. Therefore they can be used, for example, to calculate curvatures in a relatively easy way.  In \cite{kubo2016} the same ideas are used successfully in the pseudo-Riemannian case.  

It would be then natural to try to use the same ideas for symplectic Lie groups. In a similar way as before, the study of symplectic Lie groups reduces  
to the study of \textit{symplectic Lie algebras} $(\mathfrak{g}, \omega)$, that is Lie algebras $\mathfrak{g}$ endowed with nondegenerate closed $2$-forms (or equivalently  two-cocycles $\omega \in Z^2(\mathfrak{g})$). We study the moduli space of left-invariant nondegenerate $2$-forms, which is the orbit space of the action of $\mathbb{R}^{\times} \textrm{Aut} (\mathfrak{g})$ on the space $\Omega(\mathfrak{g})$ of nondegenerate $2$-forms on $\mathfrak{g}$. Then we derive a procedure to obtain Milnor-type theorems. In the setting of symplectic Lie groups a Milnor-type theorem gives a symplectic basis for each nondegenerate $2$-form in $\mathfrak{g}$, such that the bracket relations between the elements of the basis are given in terms of a small number of parameters. We show how these theorems can be used to search for $2$-forms that are closed. We also use our method to show the existence of some particular subalgebras. These ideas are developed in Section~\ref{moduli}. We would like to emphasize that the approach presented in this paper is different from that in the existing literature. In addition, the method can be used, at least theoretically, for any Lie algebra.

As a first application of this method, in Section~\ref{examples}, we study two particular Lie algebras: the Lie algebra of the real hyperbolic space $\mathfrak{g}_{\mathbb{R}\mathrm{H}^{2n}} $, and the direct sum of the $3$-dimensional Heisenberg Lie algebra and the abelian Lie algebra $\mathfrak{h}^3\bigoplus \mathbb{R}^{2n-3}$. These Lie algebras have particularly big automorphism groups: they are parabolic subgroups of $\textrm{GL}(2n,\mathbb{R})$. Therefore, we expect the corresponding moduli spaces to be small, which is a reason of our choice of these Lie algebras. In fact, in \cite{wolf} Wolf proved that for the actions of parabolic subgroups on symmetric spaces of reductive type (which is our case), the number of orbits is finite. For the first Lie algebra, we will show in Section~\ref{Group1} that the action considered is transitive, and thus the moduli space consists of just one element. Correspondingly, we obtain that there is only one symplectic structure when $n=1$ and no symplectic structure if $n>1$. For the second Lie algebra we will show in Section~\ref{Group2} that the action considered has at most five orbits (depending on $n$). Correspondingly, we obtain that for the second Lie algebra there is only one symplectic structure for all $n>0$. Furthermore we show that the second Lie algebra admits Lagrangian ideals for any $n>0$.  

As a tool for studying the moduli space we also obtain in Section~\ref{Symplecticdecomposition} a slight modification of a decomposition theorem of symplectic matrices called symplectic QR decomposition.

The authors would like to thank Takayuki Okuda, Akira Kubo, Yuichiro Taketomi, Kaname Hashimoto, Yuji Kondo and Masahiro Kawamata for helpful comments. 

\section{Preliminaries}\label{Theory}
In this section, we recall some basic notions on left-invariant symplectic forms on Lie groups, and Lagrangian normal subgroups.
\subsection{Left-invariant symplectic $2$-forms}Let $G$ be a simply connected Lie group with dimension $2n$ and $\mathfrak{g}$ its corresponding Lie algebra. We are interested in the space of left-invariant symplectic forms on $G$:
\[\left\{ \omega(\cdot{,}\cdot) \in \bigwedge\nolimits^{\!2} T^*G \mid \omega^n\neq 0, \text{left-invariant}, d\omega=0 \right\}.
\]
This is a subset of the set of all nondegenerate left-invariant $2$-forms, denoted by
\[\Omega\left(G\right):= \left\{ \omega(\cdot{,}\cdot) \in \bigwedge\nolimits^{\!2} T^*G \mid \omega^n\neq 0, \text{left-invariant}  \right\}.
\]
For this set we have the following natural equivalence relation.
\begin{definition}\label{symplectomorphism}
Let $\omega_1$,$\omega_2\in\Omega\left(G\right)$. Then, $(G,\omega_1)$ and $(G,\omega_2)$ are said to be \textit{equivalent up to automorphism} (resp. \textit{equivalent up to automorphism and scale}) if there exists $\phi \in \mathrm{Aut(G)}$ such that $\phi^{*}\omega_1=\omega_2$ (resp. if there exist $\phi \in \mathrm{Aut(G)}$ and a constant $c\neq0$ such that $c\cdot(\phi)^{*}\omega_1=\omega_2$). 
\end{definition}
It is well known that the space $\Omega\left(G\right)$ can be identified with the space of nondegenerate $2$-forms on $\mathfrak{g}$, denoted by  

\[\Omega(\mathfrak{g}):= \left\{ \omega\left(\cdot{,}\cdot\right) \in \bigwedge\nolimits^{\!2} \mathfrak{g}^* \mid \omega^n\neq 0  \right\}.  
\]
For this set we have the following natural equivalence relation. 
\begin{definition}\label{symplectomorphismlie}
Let $\omega_1, \omega_2  \in \Omega (\mathfrak{g})$. Then, $(\mathfrak{g},\omega_1)$ and $(\mathfrak{g},\omega_2)$ are said to be \textit{equivalent up to automorphism} (resp. \textit{equivalent up to automorphism and scale}) if there exists $\phi \in \mathrm{Aut(\mathfrak{g})}$ such that $\phi^{*}\omega_1=\omega_2$ (resp. if there exist $\phi \in \mathrm{Aut(\mathfrak{g})}$ and a constant $c\neq0$ such that $c\cdot(\phi)^{*}\omega_1=\omega_2$). 
\end{definition}
When the Lie group is simply connected both notions Definition~\ref{symplectomorphism} and Definition~\ref{symplectomorphismlie} of equivalence coincide. This fact allows us to work at the Lie algebra level. 
\begin{remark}
If $(S,\omega_1)$ and $(S,\omega_1)$ are symplectic manifolds and there exists $\phi \in \mathrm{Diff}(S)$ such that $\phi^*\omega_1=\omega_2$, then $(S,\omega_1)$ and
$(S,\omega_2)$ are said to be symplectomorphically equivalent and $\phi$ is called a symplectomorphism. Notice that the equivalence relation in Definition~\ref{symplectomorphism} (and the corresponding notion in Definition~\ref{symplectomorphismlie}) is stronger, but this would be the usual notion of equivalence in symplectic Lie groups. In fact, in the context of symplectic Lie groups, the map in Definition \ref{symplectomorphism} or Definition \ref{symplectomorphismlie} is also sometimes called a symplectomorphism. 
\end{remark}

Remember that a symplectic vector space is a pair $(V,\omega)$, where $V$ is a vector space and $\omega$ is a nondegenerate $2$-form. For every $\omega_{\mathfrak{g}} \in \Omega(\mathfrak{g})$ the pair $(\mathfrak{g},\omega_{\mathfrak{g}})$ is a symplectic vector space.  The next is a well known fact.

\begin{prop}\label{closedcondition} 
Let $\omega_{\mathfrak{g}} \in \Omega(\mathfrak{g})$, and $\omega_G\in \Omega(G)$ be the corresponding $2$-form on the Lie group. Then $\omega_G$ is closed if and only if $\omega_{\mathfrak{g}}$ satisfies, for all $x,y,z \in \mathfrak{g}$
\[
d\omega_{\mathfrak{g}}(x,y,z):=\omega_{\mathfrak{g}} (x,[y,z])+ \omega_{\mathfrak{g}} (z,[x,y])+\omega_{\mathfrak{g}} (y,[z,x])=0. 
\]
\end{prop}
A $2$-form $\omega_{\mathfrak{g}} \in \Omega(\mathfrak{g})$ that satisfies the previous property will be called a \textit{closed $2$-form or symplectic form} on the Lie algebra $\mathfrak{g}$. 

\begin{remark}
The previous condition can be expressed in terms of the cohomology of Lie algebras. One knows that $\omega_{\mathfrak{g}}$ is closed if and only if $\omega \in Z^2(\mathfrak{g})$, where $Z^2(\mathfrak{g})$ is the set of 2-cocyles in the trivial representation over $\mathbb{R}$. 
\end{remark}{}
From the theory of homogeneous spaces we also have the identification
\[
\Omega (\mathfrak{g}) \cong \mathrm{GL(2n,\mathbb{R})} / \mathrm{Sp_n}(\mathbb{R}).
\]
Here we identify $\mathfrak{g}\cong \mathbb{R}^{2n}$, and then the general linear group $\mathrm{GL}(2n,\mathbb{R})$ acts transitively on $\Omega\left(\mathfrak{g}\right)$ by 
\[
g.\omega(\cdot{,}\cdot)=\omega(g^{-1}(\cdot) {,} g^{-1}(\cdot)) \; \; \; \forall g\in \mathrm{GL}(2n,\mathbb{R}).
\] 
We also recall that $\mathrm{Sp_n}(\mathbb{R})$ is the symplectic group, that is the group of linear maps which preserve the canonical symplectic $2$-form $\omega_0$ in $\mathbb{R}^{2n}$. If $\lbrace e_1,\ldots,e_{2n}\rbrace$ is the canonical basis in $\mathbb{R}^{2n}$ and $\lbrace \varepsilon^1,\ldots,\varepsilon^{2n}\rbrace$ is the corresponding dual basis, then the canonical symplectic $2$-form is given by 
\[
\omega_0:=\varepsilon^1\wedge \varepsilon^{n+1}+\cdots+\varepsilon^n\wedge \varepsilon^{2n}.
\]
Then the group $\mathrm{Sp_n}(\mathbb{R})$ can be described as 
\[
\mathrm{Sp_n}(\mathbb{R}):=\left\{ A \in \mathrm{GL}(2n,\mathbb{R}) \mid {}^tAJA=J \right\}
\]
where $J:= \begin{pmatrix}
0 & I_n \\
-I_n & 0
\end{pmatrix}$.
\begin{remark}
The symplectic group $\mathrm{Sp_n}(\mathbb{R})$ is itself a Lie group, but this Lie group is not the definition of ``symplectic Lie group''.
\end{remark}

We wish to study the set $\Omega\left(\mathfrak{g}\right)$ and use Proposition~\ref{closedcondition} to search for $2$-forms that are closed, but this set can be rather big so we introduce the concept of the \textit{moduli space} in Section~\ref{moduli}. 

\subsection{Subalgebras of symplectic Lie algebras}
In this subsection, we recall the definition of some particular subalgebras of symplectic Lie algebras and make a brief comment about their importance. For details, we refer to \cite{SymplecticLieGroups} and references within. 

In the context of symplectic Lie algebras, isotropic and Lagrangian subalgebras are of particular importance (at the Lie group level these correspond to isotropic and Lagrangian subgroups, respectively).  Given a subspace $W$ of a symplectic vector space $(V,\omega)$,  recall that the orthogonal space $W^{\perp}$ with respect to $\omega$ is defined by 
\[W^{\perp}:=\left\{ v \in V \mid \omega(v,w)=0  \; \text{for all} \; w \in W   \right\}.
\]
\begin{definition}
Let $(\mathfrak{g},\omega)$ be a symplectic Lie algebra. A subalgebra $\mathfrak{l}\subset g$ is said to be \textit{isotropic} (resp. \textit{Lagrangian}) if it satisfies $\mathfrak{l}\subset \mathfrak{l}^{\perp}$ (resp. $\mathfrak{l}=\mathfrak{l}^{\perp}$). 
\end{definition}

For example, the theory of \textit{symplectic reduction} with respect to isotropic ideals has been used to study the structure of symplectic Lie algebras. Lie algebras that do not have a symplectic reduction (called \textit{irreducible}) or those that can be reduced to a trivial Lie algebra (\textit{called completely reducible}) are of particular interest. In particular, a classifications exists for irreducible symplectic Lie algebras. Lagrangian ideals also appear in connection with flat Lie algebras, as every symplectic Lie algebra with a Lagrangian ideal can be constructed as Lagrangian extension of a flat Lie algebra. 

In low dimensions several results with respect to the existence of isotropic and Lagrangian subalgebras are known, but in general the question is still not satisfactorily answered. 

\section{The moduli spaces}\label{moduli}
In this section, we define the moduli space of left-invariant nondegenerate $2$-forms on a Lie group, and formulate the procedure to obtain a kind of generalization of the Milnor frames.
\subsection{The definition}
Consider the automorphism group of $\mathfrak{g}$ defined by 
\[
\mathrm{Aut}(\mathfrak{g}):=\left\{ \phi \in \mathrm{GL}(2n,\mathbb{R}) \mid \phi [\cdot{,}\cdot] = [\phi (\cdot), \phi(\cdot)] \right\}.
\]
Also define $\mathbb{R}^{\times}:=\mathbb{R} \setminus 0$. Then we can consider the set 
\[
\mathbb{R}^{\times}\mathrm{Aut}(\mathfrak{g}):=\left\{ \phi \in \mathrm{GL}(2n,\mathbb{R}) \mid \phi \in \mathrm{Aut}(\mathfrak{g}), c \in \mathbb{R}^{\times} \right\},
\]
which is a subgroup of $\mathrm{GL}(2n,\mathbb{R})$. Hence it naturally acts on $\Omega\left(\mathfrak{g}\right)$. We can then consider the orbit space of this action.

\begin{definition}\label{mudulispace}
The orbit space of the action of $\mathbb{R}^{\times}\mathrm{Aut}(\mathfrak{g})$ on $\Omega\left(\mathfrak{g}\right) $ will be called the \textit{moduli space of left-invariant nondegenerate $2$-forms} and will be denoted by 
\[
\mathfrak{P}\Omega(\mathfrak{g}):=\mathbb{R}^{\times}\mathrm{Aut}(\mathfrak{g})\setminus \Omega (\mathfrak{g}):=\left\{ \mathbb{R}^{\times}\mathrm{Aut}(\mathfrak{g}).\omega \mid \omega \in \Omega (\mathfrak{g}) \right\}.
\]
\end{definition}
One can easily see that, if $\omega_1, \omega_2 \in \Omega (\mathfrak{g})$ are in the same $\mathbb{R}^{\times}\mathrm{Aut}(\mathfrak{g})$-orbit, then they are equivalent up to automorphism and scale. Therefore there is a surjection from the moduli space $\mathfrak{P}\Omega(\mathfrak{g})$ onto the quotient space
\[
 \Omega (\mathfrak{g}) / \text{``up to automorphism and scale''}.
\]
This would be not bijective, since $\omega$ and $-\omega$ are possibly not in the same $\mathbb{R}^{\times}\mathrm{Aut}(\mathfrak{g})$-orbit. In other words, the moduli space $\mathfrak{P}\Omega(\mathfrak{g})$ gives a finer partition than the latter quotient space. Note that the equivalence up to automorphism and scale preserves the closedness of 2-forms. Therefore, in order to search closed $2$-forms, we have only to consider this quotient space, and also the moduli space.

In the latter sections, instead of studying $\Omega (\mathfrak{g})$ directly we will focus on studying the moduli space: we want to find orbits that correspond to closed $2$-forms. As mentioned in the introduction, inspired in \cite{TakahiroHashinaga2016} and \cite{kubo2016} we introduce in the next subsection a strategy to study the moduli space and to find the orbits of closed $2$-forms. 
\subsection{Milnor frames procedure} 
Remember that given a symplectic vector space $(V,\omega)$  with $\mathrm{dim}V=2n$, we can always choose a basis $\left\{ x_1,\ldots,x_{2n} \right\}$ of $V$ such that for $i<j$
\[   
\omega(x_i,x_j) = 
     \begin{cases}
       1 & (\text{if} \; j=i+n), \\
       0 & (\text{all other cases}).
     \end{cases}
\]
This basis is called a symplectic basis.

Remember that we denote by $\omega_0$ the canonical symplectic form on a Lie algebra $\mathfrak{g}\cong \mathbb{R}^{2n}$. Let $\lbrace e_1,\ldots, e_{2n} \rbrace$ denote the corresponding symplectic basis.  To simplify the notation let us denote the orbit of $\mathbb{R}^{\times}\mathrm{Aut}(\mathfrak{g})$ through $\omega \in  \Omega(\mathfrak{g})$ by
\[
\left[ \omega \right]:=\left(\mathbb{R}^{\times}\mathrm{Aut}(\mathfrak{g})\right).\omega:=\left\{ \phi.\omega \mid \phi \in \mathbb{R}^{\times}\mathrm{Aut}(\mathfrak{g}) \right\}.
\]
\begin{definition}
A subset $U\subset \mathrm{GL}(2n,\mathbb{R})$ is called a \textit{set of representatives} of $\mathfrak{P}\Omega(\mathfrak{g})$ if it satisfies
\[
\mathfrak{P}\Omega(\mathfrak{g})=\left\{ \left[ h.\omega_0\right] \mid h \in U \right\}.
\]
\end{definition}
\begin{remark}
Of course the set of representatives is not unique. In practice we want the set of representatives to be as small as possible.
\end{remark}
Let $\left[\left[ g \right]\right]$ denote the double coset of $g \in \mathrm{GL}(2n,\mathbb{R}) $ defined by
\[
\left[\left[ g \right]\right]:= \mathbb{R}^{\times}\mathrm{Aut}(\mathfrak{g})\,g\, \mathrm{Sp}(2n,\mathbb{R}):= \left\{ \phi g s \mid \phi \in \mathbb{R}^{\times}\mathrm{Aut}(\mathfrak{g}), s\in \mathrm{Sp}_n(\mathbb{R}) \right\}.
\]
By standard theory of double coset spaces, we have a criterion for a set $U$ to be a set of representatives (we refer to \cite{TakahiroHashinaga2016}).
\begin{lemma}
Let $U\subset \mathrm{GL}(2n,\mathbb{R})$, and assume that for every $g\in \mathrm{GL}(2n,\mathbb{R}) $ there exists $h\in U$ such that $h \in \left[\left[ g \right]\right] $. Then $U$ is a set of representatives of $\mathfrak{P}\Omega(\mathfrak{g})$.
\end{lemma}
Now we state a theorem for obtaining Milnor type frames in the symplectic case. 
\begin{theorem}\label{MilnorThm}
Let $U$ be a set of representatives of $\mathfrak{P}\Omega(\mathfrak{g})$. Then for every $\omega \in \Omega (\mathfrak{g})$ there exist $k>0$, $\phi \in\mathrm{Aut}(\mathfrak{g})$ and $h\in U$ such that $\lbrace \phi ge_1,\ldots,\phi ge_{2n} \rbrace$ is a symplectic basis with respect to $k\omega$.
\end{theorem}
\begin{proof}
The proof is similar to \cite{TakahiroHashinaga2016}, Theorem 2.4. Take $\omega \in \Omega (\mathfrak{g})$. Let $\omega_0$ be the canonical symplectic form on $\mathfrak{g}$. From the definition of a set of representatives there exists $h\in U$ such that $\left[ \omega\right]=\left[h.\omega_0\right]$. Hence, there exist $c \in \,\mathbb{R}^{\times}$ and $\phi \in\mathrm{Aut}(\mathfrak{g})$ such that $\omega=(c\phi h). \omega_0$. Then we can easily check that  $\lbrace \phi he_1,\ldots,\phi he_{2n} \rbrace$ is a symplectic basis with respect to $c^2\omega$. Note that $k:=c^2>0$, which completes the proof. \end{proof}

The basis obtained in this theorem will be called \textit{Milnor frames}. Notice that if $U$ has a small number of parameters, the bracket relations of the Milnor frames will also be given in terms of a small set of parameters. In such cases, it becomes much easier to search for closed $2$-forms inside of $\Omega(\mathfrak{g})$. In the same way Milnor frames can also be useful to search for isotropic or Lagrangian subalgebras. We will show the previous procedure in two concrete examples in Section~\ref{examples}. Before working some concrete examples we introduce a tool that will be useful to calculate a smaller set of representatives and to obtain nice Milnor Frames.

\section{QR symplectic decomposition}\label{Symplecticdecomposition}

In order to obtain a nice set of representatives, 
it is useful to have general results for decomposing matrices using symplectic matrices.  
Some of the known results can be seen in \cite{BENNER2018407} and \cite{BUNSEGERSTNER198649}. In this section we obtain a slight modification.

First of all we set up some notations. 
We denote by $\mathrm{M}_{a \times b}(\mathbb{R})$ the set of all $a \times b$ real matrices, by $I_k$ the $k\times k $ identity matrix, and by $_{}^tM$ the transpose of a matrix $M$.

For $M \in \mathrm{M}_{2n \times 2n}(\mathbb{R})$ we frequently use block decompositions. In most of the matrices we use in this paper the size of each block can be understood from the shape of the matrix, if confusion is possible we will describe explicitly the size of the blocks. Also recall that $P \in \mathrm{GL}(n,\mathbb{R}) $ is called a permutation matrix if it induces
a permutation among the elements in the standard basis $\lbrace e_1,\ldots, e_n\rbrace$. For each
permutation $ \sigma \in S_n$ of degree $n$, we denote the corresponding permutation matrix
by $P_{\sigma}^{(n)}$, if there is no possible confusion we will simply denote it by $P_{\sigma}$ or $P$. For example, the permutation matrix corresponding to the cyclic permutation
$(n,n-1,\ldots,2,1)$ is given by
\[
P_{(n,n-1,\ldots,2,1)}^{(n)}=\left( \begin{array}{c|ccc} 
&&&  \\ 
&&I_{n-1}& \\
&&& \\\hline
1&&&
\end{array}\right).
\]

 We here 
make a list of well-known symplectic matrices 
in terms of the block decompositions.

\begin{prop} 
The next $2n \times 2n$ matrices are symplectic$:$ 
\begin{enumerate}
\item \underline{Type 1}
\begin{align*} 
\left( \begin{array}{c|c}
I_n & C \\ \hline
0 & I_n 
\end{array} \right) \; \text{or}  \; 
\left( \begin{array}{c|c}
I_n & 0  \\ \hline
C & I_n 
\end{array}\right)   \quad \mbox{with $C = {}^{t} C$ }. 
\end{align*} 
\item \underline{Type 2}
\begin{align*} 
\left( \begin{array}{c|c}
A & 0  \\ \hline
0 & {}^t A^{-1} 
\end{array} \right) 
\quad \mbox{for any $A \in \mathrm{GL}(n , \mathbb{R})$ }. 
\end{align*} 
\item \underline{Type 3}
\begin{align*} 
\left(\begin{array}{c|c}
P & 0  \\ \hline
0 & P 
\end{array}\right) \quad  \mbox{for any permutation matrix $P \in \mathrm{GL}(n , \mathbb{R})$} . 
\end{align*} 
\end{enumerate} 
\end{prop} 

Note that Type~3 is just a special case of Type~2 because ${}^tP_{\sigma}^{-1}=P_{\sigma}$, but we write it separately because we use it frequently. The next lemma gives a key step to obtain our decomposition theorem. 

\begin{lemma}
\label{lemma:1003} 
For all $M \in \mathrm{GL} (2n,\mathbb{R})$, 
there exists a symplectic matrix $S \in \mathrm{Sp_n}(\mathbb{R})$ such that 
the left-upper $n \times n$ block of $M S$ is nonsingular. 
\end{lemma} 

\begin{proof} 
Take any $M \in \mathrm{GL} (2n,\mathbb{R})$, and denote its block decomposition by 
\begin{align} 
M = \left( \begin{array}{c|c}
A & B  \\ \hline
C & D
\end{array} \right),
\end{align} 
where $A,B,C,D$ are $n\times n$ matrices. Let $r := \mathrm{rank}(A)$. 
If $r = n$, then it is enough to put $S := I_{2n}$. 
Assume that $r < n$. 
Then we have only to show that 
there exists $S \in \mathrm{Sp_n}(\mathbb{R})$ such that 
the left-upper $n \times n$ block of $M S$ has rank bigger than $r$. 

First of all, 
since $r = \mathrm{rank}(A) < n$, 
it follows from linear algebra that 
there exist $g_1,g_2 \in \mathrm{GL}(n,\mathbb{R})$ such that 
\begin{align} 
g_1 A g_2 = \left( \begin{array}{c|c}
I_r & 0 \\ \hline
0 & 0 
\end{array} \right). 
\end{align} 
We define the following matrices: 
\begin{align} 
K := \left( \begin{array}{c|c}
g_1 & 0 \\ \hline
0 & I_n 
\end{array} \right) \in \mathrm{GL}(2n,\mathbb{R}) , \quad 
S_1 := \left( \begin{array}{c|c}
g_2 & 0 \\ \hline
0 & {}^t g_2^{-1} 
\end{array} \right) \in \mathrm{Sp_n}(\mathbb{R}) . 
\end{align} 
Then one can directly see that 
\begin{align} 
M_1 := K M S_1 = 
\left( \begin{array}{cc|cc} 
I_r & 0 & \ast & \ast \\ 
0 & 0 & \alpha & \beta \\ \hline 
\ast & \ast & \ast & \ast \\ 
\ast & \ast & \ast & \ast 
\end{array} \right) , 
\end{align} 
where $\alpha \in \mathrm{M}_{n-r \times r}(\mathbb{R})$
and $\beta \in \mathrm{M}_{n-r \times n-r}(\mathbb{R})$. 

We here claim that we can assume $\beta \neq 0$ without loss of generality. 
In order to prove this, assume $\beta = 0$. 
Then one has $\alpha \neq 0$, since $M_1$ is nonsingular. 
Therefore there exists $\gamma \in \mathrm{M}_{r \times n-r}(\mathbb{R})$ such that 
$\alpha \gamma \neq 0$. 
We define a symplectic matrix of Type~2 by 
\begin{align} 
S_2 := 
\left( \begin{array}{cc|cc} 
I_r & 0 & & \\ 
- {}^t \gamma & I_{n-r} & & \\ \hline 
& & I_r & \gamma \\ 
& & 0 & I_{n-r} 
\end{array} \right) \in \mathrm{Sp_n}(\mathbb{R}) . 
\end{align} 
A direct calculation yields that 
\begin{align} 
M_1 S_2 = 
\left( \begin{array}{cc|cc} 
I_r & 0 & \ast & \ast \\ 
0 & 0 & \alpha & \alpha \gamma \\ \hline 
\ast & \ast & \ast & \ast \\ 
\ast & \ast & \ast & \ast 
\end{array} \right) . 
\end{align} 
Recall that $\alpha \gamma \neq 0$, 
which completes the proof of the claim. 

We assume $\beta \neq 0$ from now on. 
Let us consider a symplectic matrix of Type~1 defined by 
\begin{align} 
S_3 := 
\left( \begin{array}{cc|cc} 
I_r & & & \\ 
& I_{n-r} & & \\ \hline 
& & I_r & \\ 
& I_{n-r} & & I_{n-r} 
\end{array} \right) \in \mathrm{Sp_n}(\mathbb{R}) . 
\end{align} 
Then one can directly see that 
\begin{align} 
M_2 := M_1 S_3 = 
\left( \begin{array}{cc|cc} 
I_r & 0 & \ast & \ast \\ 
0 & \beta & \alpha & \beta \\ \hline 
\ast & \ast & \ast & \ast \\ 
\ast & \ast & \ast & \ast 
\end{array} \right) . 
\end{align} 
Since $\beta \neq 0$, 
the rank of the left-upper block of $M_2$ is bigger than $r$. 
This is also true for 
\begin{align} 
K^{-1} M_2 = K^{-1} M_1 S_3 = K^{-1} (K M S_1) S_3 = M (S_1 S_3) , 
\end{align} 
since the multiplication by $K^{-1}$ does not change the rank of the left-upper block. 
This completes the proof. 
\end{proof}

We are now in the position to give our decomposition theorem. 
Recall that a matrix $T = (t_{i j}) \in \mathrm{M}_{n \times n}(\mathbb{R})$ 
is said to be strictly lower triangular (resp. lower triangular)
if $t_{i j} = 0$ holds for all $i \leq j$ (resp. $i < j$ ).

\begin{theorem}\label{qrdecomp}
For all $M \in \mathrm{GL} (2n,\mathbb{R})$, 
there exist a symplectic matrix $S \in \mathrm{Sp}_n(\mathbb{R})$ and 
a strictly lower triangular matrix $T \in \mathrm{M}_{n \times n}(\mathbb{R})$ such that 
\begin{align*} 
M S = \left( \begin{array}{c|c} 
I_n & T  \\ \hline
\ast & \ast 
\end{array} \right) . 
\end{align*} 

\end{theorem}

\begin{proof}
Take any $M \in \mathrm{GL} (2n,\mathbb{R})$. 
One then knows by Lemma~\ref{lemma:1003} that there exists $S_1 \in \mathrm{Sp_n}(\mathbb{R})$ such that 
the left-upper $n \times n$ block of $M S_1$ is nonsingular. 
Denote by $A \in \mathrm{GL} (n,\mathbb{R})$ the left-upper block of $M S_1$, that is, 
\begin{align} 
M_1 := M S_1 = \left( \begin{array}{c|c} 
A & \ast  \\ \hline
\ast & \ast 
\end{array} \right) . 
\end{align} 
Then, by using a symplectic matrix of Type~2, we can transform $M_1$ into
\begin{align} 
M_2 := M_1 
\left( \begin{array}{c|c} 
A^{-1} & 0  \\ \hline
0 & {}^t A 
\end{array} \right) 
= 
\left( \begin{array}{c|c} 
I_n & B  \\ \hline
\ast & \ast 
\end{array} \right) . 
\end{align} 
For the block $B$, 
there exists a symmetric matrix $C \in \mathrm{M}_{n \times n} (\mathbb{R})$ 
such that $C + B$ is strictly lower triangular. 
Hence, by using a symplectic matrix of Type~1, we can transform $M_2$ into 
\begin{align} 
M_3 := M_2 
\left( \begin{array}{c|c} 
I_n & C \\ \hline
0 & I_n 
\end{array} \right) 
= 
\left( \begin{array}{c|c} 
I_n & C + B  \\ \hline
\ast & \ast 
\end{array} \right) , 
\end{align} 
which completes the proof. 
\end{proof} 

Our result is a modification of the so called QR unitary decomposition (see for example \cite{BENNER2018407} Theorem 2) as follows.
\begin{remark}[Unitary QR decomposition]

For any $A \in M_{2n\times 2n}(\mathbb{R})$, there always exists a decomposition of the form $A = RQ$, where 
\[Q \in \mathrm{Sp_n}(\mathbb{R}) \cap \mathrm{O}(2n,\mathbb{R}), \quad 
R=
\left( \begin{array}{c|c} 
R_{11} & R_{12} \\ \hline
* & * 
\end{array} \right),
\]
 with $R_{11}$ being lower triangular and $R_{12}$ strictly lower triangular. 
\end{remark}
Notice that in our case we are not restricted to unitary matrices and that the matrix we want to decompose is always nonsingular.

\section{Examples}\label{examples}
\subsection{The Lie group $G_{\mathbb{R}\mathrm{H^{2n}}}$}\label{Group1}

This group is the semidirect product of the abelian group $\mathbb{R}$ and $\mathbb{R}^{2n-1}$, 
where $\mathbb{R}$ acts on $\mathbb{R}^{2n-1}$ by $t.x:=e^tx$ $(t\in \mathbb{R}, x\in \mathbb{R}^{2n-1} )$. 
The corresponding Lie algebra of this group is 
$\mathfrak{g}_{\mathbb{R}\mathrm{H}^{2n}}\cong \mathbb{R}^{2n}=\mathrm{Span}\lbrace e_1,\ldots,e_{2n}\rbrace$, with bracket relations given by $[e_1,e_k]=e_{k}$ where $k=2,\ldots,2n$. For this Lie algebra it is known that (see \cite{kodama})

\begin{align} 
\mathbb{R}^{\times} \textrm{Aut} ( \mathfrak{g}_{\mathbb{R}\mathrm{H}^{2n}} ) 
=\left\{ \left( \begin{array}{c|ccc} 
\ast & 0 & \cdots & 0 \\ \hline 
\ast & & &  \\ 
\vdots & & \ast & \\ 
\ast & & &  \\ 
\end{array} 
\right) 
\in \mathrm{GL}(2n,\mathbb{R}) \right\},
\end{align} 
where the size of this block decomposition is $(1, 2n-1)$.

\begin{prop}\label{setrep1}
The action of 
$\mathbb{R}^{\times}\mathrm{Aut} ( \mathfrak{g}_{\mathbb{R}\mathrm{H}^{2n}})$ on  
$\Omega( \mathfrak{g}_{\mathbb{R}\mathrm{H}^{2n}})$ is transitive. A set of representatives $U$ for this
 action
is given by 
\[
U=\left\{ 
 I_{2n} 
\right\}.
\]

\end{prop}

\begin{proof}
Take any $g \in \mathrm{GL}(2n,\mathbb{R})$. 
By Theorem~\ref{qrdecomp} there exists a matrix $S\in \mathrm{Sp_n}(\mathbb{R}) $ such that 

\begin{align} 
[[g]] \ni g S = 
\left( \begin{array}{c|c} 
I_n & T  \\ \hline
\ast & \ast 
\end{array} \right) =: g_1 , 
\end{align} 
where $T$ is strictly lower triangular. 
Hence one has
$g_1 \in \mathbb{R}^{\times}\mathrm{Aut}\left(\mathfrak{g}_{\mathbb{R}\mathrm{H}^{2n}}\right) $, 
and also 
$g_1^{-1} \in \mathbb{R}^{\times}\mathrm{Aut}\left(\mathfrak{g}_{\mathbb{R}\mathrm{H}^{2n}}\right)$. 
This shows that
\begin{align} 
[[g]] \ni g_1^{-1} g_1 = I_{2n} , 
\end{align} 
which completes the proof. 
\end{proof}

\begin{theorem}[Milnor type]\label{milnorframes1}

For all $\omega \in \Omega \left(\mathfrak{g}_{\mathbb{R}\mathrm{H}^{2n}}\right)$, there exist $ t >0$ and a symplectic basis $\lbrace x_1,\ldots,x_{2n} \rbrace \subset \mathfrak{g}_{\mathbb{R}\mathrm{H}^{2n}}$ with respect to $t\omega$ such that
\[
  [x_1,x_k]=x_{k}  \; \; \; \text{for} \; k=2,\ldots,2n. 
  \]
  
\end{theorem}
\begin{proof}
By Proposition~\ref{setrep1}, a set of representatives $U$ for the 
action of 
$\mathbb{R}^{\times} \textrm{Aut} ( \mathfrak{g}_{\mathbb{R}\mathrm{H}^{2n}} )$ 
on 
$\Omega( \mathfrak{g}_{\mathbb{R}\mathrm{H}^{2n}} )$ is given just by the identity $I_{2n}$. Take any $\omega \in \Omega(\mathfrak{g}_{\mathbb{R}\mathrm{H}^{2n}})$. Then, by Theorem~\ref{MilnorThm} there exists $\phi\in\mathrm{Aut}(\mathfrak{g}_{\mathbb{R}\mathrm{H}^{2n}})$ such that $\lbrace x_1:=\phi e_1, \ldots , x_n:=\phi e_n \rbrace$ is symplectic with respect to $t\omega$.  Hence, we only have to check the
bracket relations among them:
\[
[x_1,x_k]=[\phi e_1,\phi e_k]=\phi [e_1,e_k]=\phi e_{k}=x_{k} \; \; \text{for} \; k=2,\ldots,2n.
\]
This completes the proof.
\end{proof}
Now we can use Proposition~\ref{closedcondition} to search for closed $2$-forms.
\begin{coro}\label{closedcondition0}
Let $\omega \in \Omega(\mathfrak{g}_{\mathbb{R}\mathrm{H}^{2n}})$. Then $d\omega=0$ if and only if $n=1$.  
\end{coro}
\begin{proof}
For any $\omega \in \Omega(\mathfrak{g}_{\mathbb{R}\mathrm{H}^{2n}})$, we can find a symplectic basis $\lbrace x_1,\ldots,x_{2n}\rbrace$ with respect to $t\omega$ such that the bracket relation are given by Theorem~\ref{milnorframes1}. 
Notice that when using the Milnor frames obtained in Theorem~\ref{milnorframes1} to search for symplectic forms, the parameter $t$ has no effect on the condition of Proposition~\ref{closedcondition}. Therefore, we can take $t=1$ without loss of generality.

If $n=1$ we automatically have $d\omega=0$. If $n>1$ we have $n+2\leq 2n$ and $x_{n+2}\in \mathfrak{g}_{\mathbb{R}\mathrm{H}^{2n}}$. Hence 
\begin{align*}
d\omega(x_1,x_2,x_{n+2})&=\omega(x_2,[x_{n+2},x_1])+\omega(x_{n+2},[x_1,x_2])\\
&=-2\omega(x_2,x_{n+2})=-2\neq 0.
\end{align*}
Therefore, no closed $2$-form corresponds to this case. This completes the proof.
\end{proof}
In terms of the above symplectic basis, one can easily see that  $\mathfrak{l}:=\text{Span}\left\{x_{2}\right\}$ defines a Lagrangian ideal of $\mathfrak{g}_{\mathbb{R}\mathrm{H}^{2n}}$.

\subsection{The Lie  group $H^3 \times \mathbb{R}^{2n-3}$ }\label{Group2}

This group is the direct product of the 3-dimensional Heisenberg Lie group $H^{3}$ 
and the abelian Lie group $\mathbb{R}^{2n-3}$. 
Denote by $\mathfrak{h}^3$ and $\mathbb{R}^{2n-3}$ the corresponding Lie algebras of 
$H^{3}$ and $\mathbb{R}^{2n-3}$, respectively. Then the corresponding Lie algebra of $H^3 \times \mathbb{R}^{2n-3}$ is  
\[
\mathfrak{h}^3\oplus \mathbb{R}^{2n-3}\cong \mathbb{R}^{2n}=\mathrm{Span}\lbrace e_1,\ldots,e_{2n}\rbrace
\]
with bracket relations given by $[e_1,e_2]=e_{2n}$. For this Lie algebra we have that (see \cite{kodama}) 
\begin{align} 
\mathbb{R}^{\times} \textrm{Aut} \left(\mathfrak{h}^3\oplus \mathbb{R}^{2n-3}\right) 
=\left\{ \left( 
\begin{array}{cc|ccc|c} 
\ast &\ast& 0 &\cdots& 0 & 0\\ 
\ast &\ast &0 &\cdots& 0 &0\\ \hline
\ast & \ast  && &  & 0\\
\vdots &\vdots&&\ast &  & \vdots\\ 
\ast &\ast&& & & 0\\ \hline
\ast &\ast&\ast&\cdots& \ast& \ast 
\end{array} 
\right) \in \mathrm{GL}(2n,\mathbb{R}) \right\},
\end{align} 
where the size of this block decomposition is $(2, 2n-3, 1)$. In order to obtain a set of representatives, we start with a calculation related to permutation matrices.
\begin{lemma}\label{simplepermlemma}
For any matrix of permutation $P_{\sigma}$, we can select
\[
K_1:=\left( \begin{array}{c|ccc} 
P_{\sigma_1}^{(2)} &&&  \\ \hline
&&& \\
&&P_{\sigma_2}^{(n-2)}& \\
&&&
\end{array} \right), \quad K_2:=
\left( \begin{array}{ccc|c} 
 &&&  \\ 
&P_{\sigma_3}^{(n-1)}&& \\
&&& \\\hline
&&&1
\end{array} \right)
\]
such that
\[
K_2P_{\sigma}{}^tK_1 \in \lbrace P_{(n,n-1,\ldots,2,1)}, I_n \rbrace.
\]
\end{lemma}
\begin{proof}
Let $P_{\sigma}$ be a permutation matrix. Recall that each column vector of $P_{\sigma}$ coincides with some $e_k$. First of all we consider the case that the matrix $P_{\sigma}$ is of the form
\[
P_{\sigma}= 
\left( \begin{array}{cc|ccc} 
 
&&&&  \\
&&&&  \\\hline
*&*&0&\cdots&0  \\
\end{array}\right).
\]
We can select a matrix $K_1$ of the above form such that 
\[
P_{\sigma}^{\prime}=P_{\sigma}{}^tK_1=\left( \begin{array}{cc|ccc} 
 
&&&&  \\
&&&&  \\\hline
1&0&0&\cdots&0  \\
\end{array}\right).
\]
Then there exists a matrix $K_2$ such that $K_2P_{\sigma}^{\prime}=P_{(n,n-1,\ldots,2,1)}$.

We next consider the remaining case, that is, 
\[
P_{\sigma}= 
\left( \begin{array}{cc|ccc} 
 
&&&&  \\
&&&&  \\\hline
0&0&\ast&\cdots& \ast  \\
\end{array}\right).
\]
In a similar way as the first case, we select a matrix $K_1$ such that 
\[P_{\sigma}^{\prime}=P_{\sigma}{}^tK_1= \left( \begin{array}{ccc|c} 
 
&&&  \\
&&&  \\\hline
0&\cdots&0&1  \\
\end{array}\right).
\]
Finally we select a matrix $K_2$ such that $K_2P_{\sigma}^{\prime}=I_n$, which completes the proof.
\end{proof}
In terms of the above lemma, we study double cosets. The following lemma gives a key step to obtain a set of representatives.
\begin{lemma}\label{lemmasetofrep}
For any $g \in \mathrm{GL}(2n,\mathbb{R})$, there exist $P \in \lbrace{I_n, P_{(n,n-1,\ldots,2,1)}^{(n)}\rbrace}$ and a strictly lower triangular matrix $T$ such that

\begin{align} 
[[g]] \ni 
\left( \begin{array}{c|c} 
I_n & T  \\ \hline
0 & P 
\end{array} \right).
\end{align}

\begin{proof}
Take $g\in \mathrm{GL}(2n,\mathbb{R})$. By Theorem~\ref{qrdecomp} there exists a matrix $S\in \mathrm{Sp_n}(\mathbb{R}) $ which changes the left-upper block to $I_n$, that is,  

\begin{align} 
[[g]] \ni g S = 
\left( \begin{array}{c|c} 
I_n & *  \\ \hline
C & \ast 
\end{array} \right) =: g_1.
\end{align} 
By using a certain matrix in $\mathbb{R}^{\times}\mathrm{Aut}\left(\mathfrak{h}^3\oplus \mathbb{R}^{2n-3}\right)$, we have
\[ 
[[g]] \ni 
\left( \begin{array}{c|c} 
I_n & 0  \\ \hline
-C & I_n 
\end{array} \right)
\left( \begin{array}{c|c} 
I_n & *  \\ \hline
C & \ast 
\end{array} \right)=
\left( \begin{array}{c|c} 
I_n & *  \\ \hline
0 & D
\end{array} \right)=:g_2.
\]
Note that the matrix $D$ is nonsingular. Then, by a result known as ``LPU decomposition'' (\cite{MatrixAnalysis} Theorem 3.5.11), there exist a lower triangular matrix $L$ and an upper matrix triangular matrix $U$ such that $P^{\prime}:=LDU$ is a permutation matrix. Therefore, we have 
\[ 
[[g]] \ni 
\left( \begin{array}{c|c} 
{}^tU &   \\ \hline
 & L 
\end{array} \right)g_2
\left( \begin{array}{c|c} 
{}^tU^{-1} &  \\ \hline
 & U 
\end{array} \right)=
\left( \begin{array}{c|c} 
I_n & \ast  \\ \hline
0 & P^{\prime} 
\end{array} \right)=:g_3.
\]
For the permutation matrix $P^{\prime}$, we can choose $K_1$ and $K_2$ given in Lemma~\ref{simplepermlemma} such that $P:=K_2P{}^tK_1 \in \lbrace{I_n, P_{(n,n-1,\ldots,2,1)}^{(n)}\rbrace}$. Note that
\[
\left( \begin{array}{c|c} 
K_1 &   \\ \hline
 & K_2 
\end{array} \right) \in  \mathbb{R}^{\times}\mathrm{Aut}\left(\mathfrak{h}^3\oplus \mathbb{R}^{2n-3}\right),
\]
because of the forms of $K_1$ and $K_2$. Hence, we have
\[
[[g]] \ni \left( \begin{array}{c|c} 
K_1 &   \\ \hline
 & K_2 
\end{array} \right)g_3\left( \begin{array}{c|c} 
K_1^{-1} &   \\ \hline
 & {}^{t}K_1
\end{array} \right)=\left( \begin{array}{c|c} 
I_n & \ast  \\ \hline
0 & P
\end{array} \right)=:g_4.
\]
Finally, by the same way as in the proof of Theorem~\ref{qrdecomp}, we can use a symplectic matrix $S^{\prime}$ of Type 1 to transform $g_4$ into the desired form 
\[
[[g]] \ni g_4S^{\prime} =\left( \begin{array}{c|c} 
I_n & T  \\ \hline
0 & P
\end{array} \right),
\]
where $T$ is strictly lower triangular. This completes the proof.
\end{proof}
\end{lemma}
We are now in the position to give a set of representatives for the action of $\mathbb{R}^{\times}\mathrm{Aut}\left(\mathfrak{h}^3\oplus \mathbb{R}^{2n-3}\right)$. In fact, it consists of two or five points, depending on $n$.

\begin{prop}\label{setrep2}
A set of representatives $U$ for the action of 
$\mathbb{R}^{\times} \mathrm{Aut} ( \mathfrak{h}^3\oplus \mathbb{R}^{2n-3})$ 
on $\Omega( \mathfrak{h}^3\oplus \mathbb{R}^{2n-3})$ is given by the matrices $(2)$ if $n=2$, and by the matrices $(1)$--$(3)$ if $n>2$:
 \begin{align*}
(1)&I_{2n}+kE_{2,n+1},  &(2)& \left( \begin{array}{c|c} 
I_n & 0  \\ \hline
0 & P_{(n,1,\ldots,n-1)}
\end{array} \right)+kE_{2,n+1}\\  
(3)& \left( \begin{array}{c|c} 
I_n & 0  \\ \hline
0 & P_{(n,1,\ldots,n-1)}
\end{array} \right)+E_{3,n+1},
 \end{align*}
  where $k \in \lbrace0,1\rbrace.$
\end{prop}
\begin{proof}
Take any $g\in \mathrm{GL}(2n,\mathbb{R})$. It is enough to show that $[[g]]$ contains one of the matrices $(1)$--$(3)$ if $n>2$ and one of the matrices $(2)$ if $n=2$.  By Lemma~\ref{lemmasetofrep}, there exist $P \in \lbrace{I_n, P_{(n,n-1,\ldots,2,1)}^{(n)}\rbrace}$ and a strictly lower triangular matrix $T$ such that
\begin{equation}
[[g]] \ni \left( \begin{array}{c|c} 
I_n & T  \\ \hline
0 & P
\end{array} \right)=:g_1. 
\end{equation}
We consider two cases according to the choice of $P$, and some subcases.

Case 1: $P=P_{(n,n-1,\ldots,2,1)}^{(n)}$. We write the matrix $T\in M_{n\times n}(\mathbb{R}) $ as 
\[
T= \left( \begin{array}{c|c} 
0 & 0  \\ \hline
* & T^{\prime}
\end{array} \right),
\]
where $T^{\prime} \in M_{(n-1)\times (n-1)}(\mathbb{R})$ is strictly lower triangular. We define some matrices: 
\[
A:=\left( \begin{array}{c|c} 
0 & 0  \\ \hline
-T^{\prime} & 0
\end{array} \right) \in M_{n\times n}(\mathbb{R}), 
   \quad D:=  \left( \begin{array}{c|c} 
I_n & A  \\ \hline
 0& I_n 
\end{array} \right) \in \mathbb{R}^{\times}\mathrm{Aut}\left(\mathfrak{h}^3\oplus \mathbb{R}^{2n-3}\right). 
\]
We use the matrix $D$ to transform the matrix $g_1$ into
\[
[[g]] \ni  Dg_1 = \left( \begin{array}{c|c} 
I_n & T+AP  \\ \hline
 0& P \end{array}
 \right)=:g_2.
 \]  
By a direct calculation of $T+AP$, one can express $g_2$ as 
\begin{equation}
g_2=\left( \begin{array}{c|c} 
I_n & 0  \\ \hline
 0& P \end{array}
 \right)+x_2E_{2,n+1}+\cdots+x_nE_{n,n+1}. \
\end{equation}
We here consider the following three subcases.

Subcase $1$--$(i)$: $x_2\neq 0$. In this case, we can transform $x_2$ into 1 as follows. Define the following two diagonal matrices:
\[
V_1:=E_{1,1}+(1/x_2)E_{2,2}+E_{3,3}+\cdots+E_{n,n} , \quad V_2=P({}^tV_1^{-1}){}^tP.
\]
Then we can take matrices which makes $x_2$ into 1, that is, 
\begin{eqnarray*}
[[g]] &\ni& \left( \begin{array}{c|c} 
V_1 &  \\ \hline
 & V_2 
\end{array} \right)g_2\left( \begin{array}{c|c} 
V_1^{-1} &  \\ \hline
 & {}^tV_1 
\end{array} \right)\\
&=&\left( \begin{array}{c|c} 
I_n & 0  \\ \hline
 0& P \end{array}
 \right)+E_{2,n+1}+x_3E_{3,n+1}+\cdots+x_nE_{n,n+1}=:g_3.
\end{eqnarray*}
Now define the matrices:
\[
T_1:=I_{n}-x_3E_{3,2}-\cdots-x_nE_{n,2}\in M_{n\times n}(\mathbb{R}), \quad T_2:=P({}^tT_1^{-1}){}^tP.
\]
We can take matrices that eliminate the rest of parameters in $g_3$, that is,
\begin{equation}
[[g]] \ni \left( \begin{array}{c|c} 
T_1 &  \\ \hline
 & T_2 
\end{array} \right)g_3\left( \begin{array}{c|c} 
T_1^{-1} &  \\ \hline
 & {}^tT_1 
\end{array} \right)=\left( \begin{array}{c|c} 
I_n & 0  \\ \hline
 0& P \end{array}
 \right)+E_{2,n+1}.
\end{equation}
This coincides with the matrix $(2)$ of $k=1$ in the statement of this proposition.

Subcase $1$--$(ii)$: $x_2=0$ and $n=2$.  In this case we immediately obtain
\begin{equation}
g_2=\left( \begin{array}{c|c} 
I_2 & 0  \\ \hline
0 & P_{(2,1)}
\end{array} \right),
\end{equation}
which coincides with the matrix $(2)$ of $k=0$ in the statement of this proposition. 

Subcase $1$--$(ii)$: $x=0$ and $n>2$. Given a vector ${}^t(x_3,\ldots,x_n)$, it is well known that there exists a matrix $U \in \mathrm{GL}(n-2,\mathbb{R})$ such that
\[
U\cdot{}^t(x_3,\ldots,x_n)={}^t(k,0,\ldots,0),
\]
where $k\in \lbrace{0,1\rbrace}$. Define the matrices:
\[
H_1:=\left( \begin{array}{c|c} 
I_2& \\\hline
&U \\
\end{array} \right)\in \mathrm{GL}(n,\mathbb{R}), \quad H_2:=P({}^tH_1^{-1}){}^tP.
\]
Now we can consider the transformation 
\[
[[g]] \ni \left( \begin{array}{c|c} 
H_1 &  \\ \hline
 & H_2 
\end{array} \right)g_2\left( \begin{array}{c|c} 
H_1^{-1} &  \\ \hline
 & {}^tH_1 
\end{array} \right)=:g_3.
\]
The effect of the previous transformation is just the action of $U$ on the vector ${}^t(x_3,\ldots,x_n)$.
Therefore one has 
\begin{equation}
g_3=\left( \begin{array}{c|c} 
I_n & 0  \\ \hline
 0& P \end{array}
 \right)+kE_{3,n+1},   
\end{equation}
where $k\in\lbrace0,1 \rbrace$.  This coincides with the matrix $(2)$ of $k=0$, or with the matrix $(3)$ in the statement of this proposition. 

Case 2: $P=I_{n}$. We consider the following two subcases. 

Subcase $2$--$(i)$: $n=2$. We show that this case is equivalent to Case 1. We have
\[
g_1=\left( \begin{array}{c|c} 
I_2 & T  \\ \hline
 0& I_2 \end{array} \right).
\]
Define the matrices:
\[
R=\left( \begin{array}{c|c} 
P_{(2,1)} &   \\ \hline
 & I_2 
\end{array} \right)
 \in \mathbb{R}^{\times}\mathrm{Aut}\left(\mathfrak{h}^3\oplus \mathbb{R}^{1}\right), \quad  S=\left( \begin{array}{c|c} 
P_{(2,1)} &   \\ \hline
 & P_{(2,1)}
\end{array} \right) \in \mathrm{Sp_2}(\mathbb{R}).
\]
Now we can consider the transformation
\[
[[g]] \ni Rg_1S=\left( \begin{array}{c|c} 
I_2 & *  \\ \hline
 0& P_{(2,1)} 
\end{array} \right)=:g_2.
\]
In a similar way to the final part of Lemma~\ref{lemmasetofrep} we can use a symplectic matrix $S^{\prime}$ of Type 1 to transform $g_2$ into 
\[
[[g]] \ni g_2S^{\prime}=\left( \begin{array}{c|c} 
I_n & T  \\ \hline
 0& P_{(2,1)} 
\end{array} \right),
\]
where $T$ is strictly lower triangular. This matrix corresponds to Case 1. Therefore, for the case of $n=2$, we have obtained a set of representatives consisting of two points.

Subcase $2$--$(ii)$: $n>2$. We start from the matrix $g_1$, and consider the matrix $T=(t_{i,j})\in M_{n\times n}(\mathbb{R}) $. Let us put

\[
T^{\prime}:=T-t_{2,1}E_{2,1}.
\]
Since $T$ is lower triangular, one has
\[
 D^{\prime}:=  \left( \begin{array}{c|c} 
I_n & -T^{\prime}  \\ \hline
 0& I_n 
\end{array} \right) \in \mathbb{R}^{\times}\mathrm{Aut}\left(\mathfrak{h}^3\oplus \mathbb{R}^{2n-3}\right). 
\]
We use the matrix $D^{\prime}$ to transform the matrix $g_2$ into
\[[[g]] \ni D^{\prime}g_1=I_{2n}+t_{2,1}E_{2,n+1}=:g_2.
\]
Finally, in a similar way to Case 1, if $t_{2,1} \neq 0$ we can transform it easily into 1. After this transformation we obtain
\begin{equation}
[[g]] \ni I_{2n}+kE_{2,n+1},
\end{equation}
where $k\in \lbrace0,1\rbrace$. This coincides with the matrix $(1)$ in the statement of this proposition, which completes the proof.
\end{proof}

Given this set of representatives we can give a Milnor-type theorem.
\begin{theorem}[Milnor-type]\label{milnorframes2}

For all $\omega \in \Omega(\mathfrak{h}^3\bigoplus \mathbb{R}^{2n-3})$, there exist $t>0$, $ k\in\lbrace0,1\rbrace$ and a symplectic basis $\lbrace x_1,\ldots,x_{2n}\rbrace$ of $\mathfrak{h}^3\bigoplus \mathbb{R}^{2n-3}$ with respect to $t\omega$ such that the bracket relations are given by one of the following:
\begin{enumerate}
    \item $[x_1,x_2 ]= x_{2n}, \; [x_1,x_{n+1} ]= kx_{2n} \; (for \; n>2)$,
    \item $[x_1,x_2]= x_{n+1}-kx_{2}, \;[x_1,x_{n+1}]= kx_{n+1}-k^2x_{2}$, or
    \item $[x_1,x_2 ]= x_{n+1}-x_3 \; (for \; n>2)$.
    
\end{enumerate}
\begin{proof}
 Let $\lbrace e_1, \ldots , e_{2n} \rbrace$ be the canonical basis of $\mathfrak{h}^3\bigoplus \mathbb{R}^{2n-3}$, whose bracket relation is given by $[e_1,e_2]=e_{2n}$. In Proposition~\ref{setrep2} we obtained a set of representatives $U$ for the action of 
$\mathbb{R}^{\times} \textrm{Aut} ( \mathfrak{h}^3\oplus \mathbb{R}^{2n-3})$ 
on $\Omega( \mathfrak{h}^3\oplus \mathbb{R}^{2n-3})$. Take any $\omega \in \Omega(\mathfrak{h}^3\bigoplus \mathbb{R}^{2n-3})$. Then it follows from Theorem~\ref{MilnorThm} that there exist $u \in  U, t>0$ and $\phi \in \mathrm{Aut}(\mathfrak{h}^3\bigoplus \mathbb{R}^{2n-3})$ such that $\lbrace x_1:=\phi ue_1, \ldots , x_n:=\phi ue_n \rbrace$ is symplectic with respect to $t\omega$. Hence, we only have to check the
bracket relations among them. We study each case corresponding to the five matrices in $U$ case by case.

Case 1: we consider the case of $u=I_{2n}+kE_{2,n+1}$ with $k\in \lbrace0,1\rbrace$ and $n>2$, which corresponds to the matrices $(1)$ in Proposition~\ref{setrep2}. In this case one has
\[
ue_i = 
      \begin{cases}
       ke_2+e_{n+1} &\quad (i=n+1),\\
       e_i & \quad (i\neq n+1). \\
      \end{cases}
\]
Therefore, among $\lbrace ue_1, \ldots , ue_{2n} \rbrace$, the only possible nonzero bracket relations are 
\begin{align*}
    [ue_1,ue_2]&=[e_1,e_2]=e_{2n}=ue_{2n}, \\
    [ue_1,ue_{n+1}]&=[e_1,ke_2+e_{n+1}]=kue_{2n}.
\end{align*}
By applying the automorphism  $\phi$ to both sides, we obtain
\begin{align*}
    [x_1,x_2]&=[\phi ue_1,\phi ue_2]=\phi [ue_1,ue_2]=\phi ue_{2n}=x_{n}, \\
    [x_1,x_{n+1}]&=[\phi ue_1,\phi ue_{n+1}]=\phi[ ue_1,ue_{n+1}]=\phi kue_{2n}=kx_{2n}.
\end{align*}
Therefore the bracket relations are given by $(1)$ in the statement of this theorem.

Case 2: we consider the case of
\[ \left( \begin{array}{c|c} 
I_n & 0  \\ \hline
0 & P_{(n,1,\ldots,n-1)}
\end{array} \right)+kE_{2,n+1} \quad (\mathrm{with} \; \; k\in\lbrace0,1\rbrace),\]
which corresponds to the matrices $(2)$ in Proposition~\ref{setrep2}. In this case one has
\[
ue_i = 
      \begin{cases}
       e_i & \quad (i\leq n),\\
       ke_2+e_{2n} & \quad (i=n+1),\\
       e_{i-1} & \quad (i>n+1).
      \end{cases}
\]
Therefore, among $\lbrace ue_1, \ldots , ue_{2n} \rbrace$, the only possible nonzero bracket relations are 
\begin{align*}
    [ue_1,ue_2]&=[e_1,e_2]=e_{2n}=ue_{n+1}-kue_2, \\
    [ue_1,ue_{n+1}]&=[e_1,ke_2+e_{n+1}]=kue_{n+1}-k^2ue_2.
\end{align*}
By applying the automorphism  $\phi$ to both sides, we obtain the bracket relations given by $(2)$ in the statement of this theorem.

Case 3: we consider the case of $n>2$ and 
\[ \left( \begin{array}{c|c} 
I_n & 0  \\ \hline
0 & P_{(n,1,\ldots,n-1)}
\end{array} \right)+E_{3,n+1},\]
which corresponds to the matrices $(3)$ in Proposition~\ref{setrep2}. In this case one has
\[
ue_i = 
      \begin{cases}
       e_i & \quad (i\leq n),\\
       e_3+e_{2n} & \quad (i=n+1),\\
       e_{i-1} & \quad (i>n+1).
      \end{cases}
\]
Therefore, among $\lbrace ue_1, \ldots , ue_{2n} \rbrace$, the only possible nonzero bracket relations are 
\[
    [ue_1,ue_2]=[e_1,e_2]=e_{2n}=ue_{n+1}-ue_3.
\]
By applying the automorphism  $\phi$ to both sides, we obtain the bracket relations given by $(3)$. This completes the proof.
\end{proof}  
\end{theorem}

Now we can use Proposition~\ref{closedcondition} to search for closed $2$-forms and, obtain the existence and uniqueness result as follows.

\begin{coro}\label{symplecticstructures2}
There exists a unique symplectic structure on $\mathfrak{h}^3\bigoplus \mathbb{R}^{2n-3}$ up to automorphism and scale. In fact, $\omega \in \Omega(\mathfrak{h}^3\oplus \mathbb{R}^{2n-3})$ satisfies $d\omega=0$ if and only if there there exist $t>0$ and a symplectic basis $\lbrace x_1,\ldots,x_{2n} \rbrace \subset \mathfrak{h}^3\oplus \mathbb{R}^{2n-3}$ with respect to $t\omega$ such that
\[
[x_1,x_2 ]= x_{n+1}.
\]
\end{coro}
\begin{proof}
For any $\omega \in \Omega(\mathfrak{h}^3\oplus \mathbb{R}^{2n-3})$ we can find a symplectic basis $\lbrace x_1,\ldots,x_{2n}\rbrace$ with respect to $t\omega$ such that the bracket relations are given by $(1)$--$(3)$ in Theorem~\ref{milnorframes2}. As we did in Proposition~\ref{closedcondition0} we take $t=1$, without loss of generality.  We check the closed condition for each case.  

Bracket relations $(1)$. In this case we have $n>2$, and 
\[
[x_1,x_2 ]= x_{2n}, \; \; \; [x_1,x_{n+1}]=kx_{2n}.
\]
Note that $[x_1,x_n]=0$ since $n\neq 2$. Then one can see that 
\begin{align*}
d\omega(x_1,x_2,x_n)&=\omega(x_1,[x_2,x_n])+\omega(x_n,[x_1,x_2])+\omega(x_2,[x_n,x_1])\\
&=\omega(x_n,x_{2n})=1\neq 0.
\end{align*}
Therefore, no closed $2$-form corresponds to this case.

Bracket relations $(2)$. Recall that it satisfies
\[
[x_1,x_2]= x_{n+1}-kx_{2}, \; \; \; [x_1,x_{n+1}]= kx_{n+1}-k^2x_{2}.
\]
One can easily see that
\[
d\omega(x_1,x_{n+1},x_{n+2})=\omega(x_{n+2},[x_1,x_{n+1}])=k^2.
\]
Hence we have only to consider the case of $k=0$. In this case, the bracket relations reduce to 
\[
[x_1,x_2 ]= x_{n+1}.
\]
For this bracket relations it is easy to check that $d\omega=0$. In particular, for each $j \in \lbrace{3,\ldots, 2n\rbrace}$, we have   
\[
d\omega(x_1,x_{2},x_j)=\omega(x_j,[x_1,x_2])=\omega(x_j,x_{n+1})=0.
\]
This bracket corresponds to a closed $2$-form, as desired.

Bracket relations $(3)$. In this case we have $n>2$, and
\[
[x_1,x_2 ]= x_{n+1}-x_3.
\]
Then one can easily see that
\[
d\omega(x_1,x_2,x_{n+3})=\omega(x_{n+3},[x_1,x_2])=1\neq 0.
\]
Therefore, no closed $2$-form corresponds to this case. This completes the proof.
\end{proof}

In terms of the above symplectic basis, one can easily see the existence of a Lagrangian normal subgroup in $H^3\times \mathbb{R}^{2n-3}$, equipped with a left-invariant symplectic structure.
\begin{coro}
For every symplectic structure on $\mathfrak{h}^3\oplus \mathbb{R}^{2n-3}$ there exists a Lagrangian ideal.

\begin{proof}
Let $\omega$ be a symplectic structure on $\mathfrak{h}^3\oplus \mathbb{R}^{2n-3}$. Then by Corollary~\ref{symplecticstructures2}, there exist $t>0$ and a symplectic basis $\lbrace x_1,\ldots,x_{2n} \rbrace$ with respect to $t\omega$ such that
\[
[x_1,x_2]=x_{n+1}.
\]
Define $\mathfrak{l}=\lbrace x_2,\ldots, x_n, x_{n+1}\rbrace$. Then one can easily check that $\mathfrak{l}$ is a Lagrangian ideal.
\end{proof}
\end{coro}

\end{document}